\documentclass[12pt]{amsart}
\usepackage{amsmath}
\usepackage{amstext}
\usepackage{amsfonts}
\usepackage{amsmath}
\usepackage{amssymb}
\usepackage{amsthm}
\usepackage{amsrefs}
\usepackage{bbm}
\usepackage[pdftitle={Thompson's theorem for II1 factors},pdfauthor={Matthew Kennedy and Paul Skoufranis},bookmarks=true]{hyperref}

\numberwithin{equation}{section}

\newtheorem*{theorem*}{Theorem \ref{thm:thompson-intro}}

\theoremstyle{plain}
\newtheorem{thm}{Theorem}[section]

\newtheorem{prop}[thm]{Proposition}
\newtheorem{lem}[thm]{Lemma}

\theoremstyle{definition}
\newtheorem{defn}[thm]{Definition}
\newtheorem{rem}[thm]{Remark}

\newcommand{\bC}{{\mathbb{C}}}

\newcommand{\bR}{{\mathbb{R}}}

\newcommand{\A}{{\mathcal{A}}}
\newcommand{\B}{{\mathcal{B}}}
\newcommand{\F}{{\mathcal{F}}}

\newcommand{\M}{{\mathcal{M}}}
\newcommand{\N}{{\mathcal{N}}}
\renewcommand{\O}{{\mathcal{O}}}

\newcommand{\U}{{\mathcal{U}}}
\newcommand{\X}{{\mathcal{X}}}
\newcommand{\Y}{{\mathcal{Y}}}
\newcommand{\Z}{{\mathcal{Z}}}

\renewcommand{\phi}{\varphi}

\newcommand{\re}{\operatorname{Re}}

\newcommand{\II}[1]{II$_1$}
\newcommand{\submaj}{\prec_w}
\newcommand{\maj}{\prec}
\newcommand{\Proj}{\operatorname{Proj}}

\DeclareMathOperator*{\sotlim}{\operatorname{s-lim}}

\begin{document}
%%%%%%%%%%%%%%%%%%%%%%%%%%%%%%%%%%%%%%%%%%%
\title[Thompson's theorem for $II_1$ factors]{Thompson's theorem for $II_1$ factors}

\author[M. Kennedy]{Matthew Kennedy}
\address{School of Mathematics and Statistics\\ Carleton University\\
Ottawa, Ontario \; K1S 5B6 \\Canada}
\email{mkennedy@math.carleton.ca}

\author[P. Skoufranis]{Paul Skoufranis}
\address{Department of Mathematics\\UCLA\\
Los Angeles, California\; 90095 \\ USA}
\email{pskoufra@math.ucla.edu}

\begin{abstract}
A theorem of Thompson provides a non-self-adjoint variant of the classical Schur-Horn theorem by characterizing the possible diagonal values of a matrix with given singular values. We prove an analogue of Thompson's theorem for \II1 factors.
\end{abstract}

\thanks{First author partially supported by a research grant from NSERC (Canada). Second author partially supported by a research grant from the NSF (USA)}
\maketitle

%%%%%%%%%%%%%%%%%%%%%%%%%%%%%%%%%%%%%%%%%%%
\section{Introduction}
%%%%%%%%%%%%%%%%%%%%%%%%%%%%%%%%%%%%%%%%%%%
%\subsection{The Schur-Horn Theorem}
%%%%%%%%%%%%%%%%%%%%%%%%%%%%%%%%%%%%%%%%%%%
The following classical theorem of Schur \cite{S1923} and Horn \cite{H1954} characterizes the possible diagonal entries of a self-adjoint matrix with given eigenvalues.

\begin{thm}[Schur-Horn] \label{thm:schur-horn}
Let $\lambda,\alpha \in \bR^n$ be vectors with $\lambda_1 \geq \cdots \geq \lambda_n$ and $\alpha_1 \geq \cdots \geq \alpha_n$. There is an $n \times n$ self-adjoint matrix with eigenvalues $\lambda$ and diagonal $\alpha$ if and only if
\begin{align}
\sum_{j = 1}^k \alpha_j &\leq \sum_{j = 1}^k \lambda_j, \quad k = 1,\ldots,n-1,\ \text{and}\label{eq:thm:schur-horn-1} \\
\sum_{j = 1}^n \alpha_j &= \sum_{j = 1}^n \lambda_j. \notag
\end{align}
\end{thm}

Mirsky \cite{M1964} asked for a variation on the Schur-Horn theorem that did not require the matrix to be self-adjoint. Specifically he asked for a characterization of the possible diagonal entries of a matrix with given singular values. This problem was eventually solved by Thompson \cite{T1977}*{Theorem 1} who proved the following result (see also \cite{S1976}).

\begin{thm}[Thompson] \label{thm:thompson}
Let $\sigma \in \bR^n$ and $\alpha \in \bC^n$ be vectors with $\sigma_1 \geq \cdots \geq \sigma_n \geq 0$ and $|\alpha_1| \geq \cdots \geq |\alpha_n|$. There is an $n \times n$ matrix with singular values $\sigma$ and diagonal $\alpha$ if and only if
\begin{subequations}
\begin{align}
\sum_{j= 1}^k |\alpha_j| &\leq \sum_{j = 1}^k \sigma_j, \quad k = 1,\ldots,n,\  \text{and}\label{eq:thm:thompson-1} \\
\sum_{j=1}^{n-1} |\alpha_j| - |\alpha_n| &\leq  \sum_{j = 1}^{n-1} \sigma_j - \sigma_n.  \label{eq:thm:thompson-2}
\end{align}
\end{subequations}
\end{thm}

A compelling case can be made that von Neumann factors of type \II1 are the best infinite dimensional analogues of matrix algebras. Recently, Ravichandran \cite{R2012} proved a version of the Schur-Horn theorem for these algebras, thereby settling a problem of Arveson and Kadison \cite{AK2006}. This result is the culmination of a great deal of work done by Argerami and Massey \cites{AM2007,AM2009,AM2013}; Dykema, Fang, Hadwin and Smith \cite{DFHS2011}; and Bhat and Ravichandran \cite{BR2011}.

The main result in this paper is a version of Thompson's theorem for \II1 factors. Before stating the result, it is helpful to first consider the classical Schur-Horn theorem and Thompson's theorem from a slightly different perspective.

Let $\lambda,\alpha \in \bR^n$ be vectors with $\lambda_1 \geq \cdots \geq \lambda_n$ and $\alpha_1 \geq \cdots \geq \alpha_n$. If (\ref{eq:thm:schur-horn-1}) holds, then we say that $\lambda$ majorizes $\alpha$ and write $\alpha \maj \lambda$.

Observe that the Schur-Horn theorem is equivalent to the statement that if $S \in \M_n$ is a diagonal matrix with diagonal $\lambda$, then $\alpha \maj \lambda$ if and only if there is a unitary $U \in \U(\M_n)$ such that the matrix $USU^*$ has diagonal $\alpha$.

Now let $\sigma \in \bR^n$ and $\alpha \in \bC^n$ be vectors with $\sigma_1 \geq \cdots \geq \sigma_n \geq 0$ and $|\alpha_1| \geq \cdots \geq |\alpha_n|$. If (\ref{eq:thm:thompson-1}) holds, then we say that $\sigma$ (absolutely) submajorizes $\alpha$ and write $\alpha \submaj \sigma$.

Observe that Thompson's theorem is equivalent to the statement that if $S \in \M_n$ is a diagonal matrix with diagonal $\sigma$, then $\alpha \submaj \sigma$ and (\ref{eq:thm:thompson-2}) if and only if there are unitaries $U,V \in \U(\M_n)$ such that the matrix $USV$ has diagonal $\alpha$.

In the setting of a \II1 factor $\M$, there are good analogues of eigenvalues and singular values. This has been known since the work of Murray and von Neumann \cite{MV1936} (see also \cites{F1982, K1983,P1985,FK1986,H1987}). Using these ideas, Hiai \cite{H1987} defined notions of majorization and submajorization for elements in $\M$, using the same notation as above (see Section \ref{sec:preliminaries}).  The appropriate analogue of the diagonal of an $n \times n$ matrix is the normal conditional expectation onto a MASA $\A$ of $\M$ (cf. \cite{SS2008}).

Ravichandran \cite{R2012}*{Theorem 5.6} proved the following version of the Schur-Horn theorem for \II1 factors.

\begin{thm}[Ravichandran] \label{thm:ravichandran}
Let $\M$ be a \II1 factor, let $\A$ be a MASA in $\M$, let $E_{\A} : \M \to \A$ denote the normal conditional expectation onto $\A$, and let $T \in \M$ and $A \in \A$ be self-adjoint elements.  Then $A \maj T$ if and only if there exists
\[
S \in \overline{\{ UTU^* \mid U \in \U(\M) \}}
\]
such that $E_{\A}(S) = A$.
\end{thm}

The following result, which is the main result of this paper, is our analogue of Thompson's theorem for \II1 factors. We do not require anything like condition (\ref{eq:thm:thompson-2}), which can be explained by the lack of minimal projections in \II1 factors.

\begin{thm} \label{thm:thompson-intro}
Let $\M$ be a \II1 factor, let $\A$ be a MASA in $\M$, let $E_{\A} : \M \to \A$ denote the normal conditional expectation onto $\A$, and let $T \in \M$ and  $A \in \A$ be arbitrary elements.  Then $A \submaj T$ if and only if there exists
\[
S \in \overline{\{ UTV \mid U,V \in \U(\M) \}}
\]
such that $E_{\A}(S) = A$. 
\end{thm}

In addition to this introduction, there are three other sections. In Section \ref{sec:preliminaries} we recall the notions of $s$-values and submajorization. In Section \ref{sec:main} we prove Theorem \ref{thm:thompson-intro}, our main result. In Section \ref{sec:relationship} we consider the relationship between the Schur-Horn theorem and Thompson's theorem for \II1 factors.

\subsection*{Acknowledgements}

We are grateful to the referee for providing us with detailed comments and suggestions for improving the presentation of this paper.

%%%%%%%%%%%%%%%%%%%%%%%%%%%%%%%%%%%%%%%%%%%
\section{Preliminaries} \label{sec:preliminaries}
%%%%%%%%%%%%%%%%%%%%%%%%%%%%%%%%%%%%%%%%%%%
\subsection{Eigenvalue and singular value functions} \label{sec:eigenvalue-singular-value-functions}
%%%%%%%%%%%%%%%%%%%%%%%%%%%%%%%%%%%%%%%%%%%
It has been known since the work of Murray and von Neumann \cite{MV1936} that there are good analogues of eigenvalues and singular values for elements in a diffuse finite von Neumann algebra (see also \cites{F1982, K1983,P1985,FK1986,H1987}).

Let $\M$ be a diffuse finite von Neumann algebra equipped with a fixed faithful normal unital trace $\tau$ (unless we specify otherwise, $\M$ will always be equipped with this trace).  For $T \in \M$ self-adjoint, we will let $p_T$ denote the \emph{spectral measure} of $T$ on $\bR$; that is, for a Borel set $B \subset \bR$, $p_T(B)$ is the spectral projection of $T$ corresponding to $B$. The \emph{spectral distribution} of $T$ is the unique Borel probability measure $m_T$ on $\bR$ satisfying
\[
\int_{\bR} s^n\, dm_T(s) = \tau(T^n), \quad n \geq 0.
\]
It follows that for every Borel set $\B \subset \bR$, we have $m_T(B) = \tau(p_T(B))$.

\begin{defn} \label{defn:eigenvalue-singular-value-functions}
Let $\M$ be a diffuse finite von Neumann algebra.
\begin{enumerate}
\item For a self-adjoint element $T \in \M$, the \emph{eigenvalue function of $T$} is defined for $s \in [0,1)$ by
\[
\lambda_s(T) = \inf \{t \in \bR \mid m_T((t,\infty)) \leq s \}.
\]
\item For an arbitrary element $T \in \M$, the \emph{singular value function of $T$} is defined for $s \in [0,1)$ by
\[
\mu_s(T) := \lambda_s(|T|).
\]
\end{enumerate}
\end{defn}

It is well-known that the singular value function of $T$ is a non-increasing, right continuous function such that $\mu_s(T) = \mu_s(|T|)$ for all $s \in [0,1)$ and $\mu_0(T) = \left\|T\right\|$ (see \cites{F1982, FK1986}).

Clearly the eigenvalue functions and singular value functions are invariant under (trace-preserving) isomorphisms of von Neumann algebras.  Note also that they do not depend on the ambient von Neumann algebra. In other words, if $\N \subset \M$ is a diffuse von Neumann subalgebra and $T \in \N$, then the eigenvalue and singular value functions of $T$ computed with respect to $\N$ are equal to the eigenvalue and singular value functions of $T$ computed with respect to $\M$.

If $P \in \M$ is a non-zero projection and $T' = PTP \in P \M P$ is the corresponding compression of $T$, then we will write $\lambda_s(T')$ and $\mu_s(T')$ for the eigenvalue function and singular value function of $T'$ as computed in $P \M P$ with respect to the normalized trace inherited from $\M$. This will always be clear from the context. 

The following result seems to be folklore, although a proof is given by Argerami and Massey in \cite{AM2007}*{Proposition 2.3}.
\begin{prop} \label{prop:special-spectral-measure}
Let $\M$ be a diffuse finite von Neumann algebra and let $T \in \M$ be self-adjoint. There is a projection-valued measure $e_T$ on $[0,1)$ such that $\tau(e_T([0,t))) = t$ for every $t \in [0,1)$ and
\[
T^n = \int_0^1 \lambda_s(T)^n\, de_T(s), \quad \forall n \geq 0.
\]
\end{prop}

For a self-adjoint element $T \in \M$, we will continue to write $e_T$ for a fixed choice of measure on $[0,1)$ obtained from Proposition \ref{prop:special-spectral-measure}.

\begin{rem} \label{rem:determined-by-spectral-distribution}
It follows from Definition \ref{defn:eigenvalue-singular-value-functions} that if $T \in \M$ is self-adjoint, then the eigenvalue function of $T$ is completely determined by the spectral distribution of $T$. On the other hand, by Proposition \ref{prop:special-spectral-measure},
\[
\tau(T^n) = \int_0^1 \lambda_s(T)^n\, ds, \quad \forall n \geq 0.
\]
Hence the spectral distribution of $T$ is completely determined by the eigenvalue function of $T$.
\end{rem}

\begin{rem} \label{rem:equidistributed-equality-of-singular-values}
For two self-adjoint elements $S,T \in \M$ we will say that $S$ and $T$ are \emph{equi-distributed} if they have identical spectral distributions $m_S$ and $m_T$ respectively. By Remark \ref{rem:determined-by-spectral-distribution}, this is equivalent to the eigenvalue functions $\lambda_s(S)$ and $\lambda_s(T)$ satisfying $\lambda_s(S) = \lambda_s(T)$ for every $s \in [0,1)$ and, if $S$ and $T$ are positive, this is equivalent to the singular value functions $\mu_s(S)$ and $\mu_s(T)$ satisfying $\mu_s(S) = \mu_s(T)$ for every $s \in [0,1)$.
\end{rem}

%%%%%%%%%%%%%%%%%%%%%%%%%%%%%%%%%%%%%%%%%%%
\subsection{Majorization and submajorization}
%%%%%%%%%%%%%%%%%%%%%%%%%%%%%%%%%%%%%%%%%%%
Notions of majorization and submajorization can be defined for both the eigenvalue function and the singular value function. We will consider majorization with respect to eigenvalue functions and submajorization with respect to singular value functions.

\begin{defn} \label{defn:majorization}
Let $\M$ be a diffuse finite von Neumann algebra.
\begin{enumerate}
\item For self-adjoint $S,T \in \M$ we will say that $T$ \emph{majorizes} $S$ and write $S \maj T$ if
\begin{align*}
\int_0^t \lambda_s(S)\, ds &\leq \int_0^t \lambda_s(T)\, ds, \quad \forall t \in [0,1),\ \text{and} \\
\int_0^1 \lambda_s(S)\, ds &= \int_0^1 \lambda_s(T)\, ds.
\end{align*}
\item For arbitrary $S,T \in \M$ we will say that $T$ \emph{(absolutely) submajorizes} $S$ and write $S \submaj T$ if
\[
\int_0^t \mu_s(S)\, ds \leq \int_0^t \mu_s(T)\, ds, \quad \forall t \in [0,1].
\]
\end{enumerate}
\end{defn}

\begin{rem} \label{rem:positives-majorization-submajorization-equivalent}
It is clear from Definition \ref{defn:eigenvalue-singular-value-functions} and Definition \ref{defn:majorization} that for arbitrary $S,T \in \M$, $S \submaj T$ if and only if $|S| \submaj |T|$. Moreover, if $S$ and $T$ are positive then $S \maj T$ if and only if $S \submaj T$ and $\tau(S) = \tau(T)$.
\end{rem}

We require the following lemma of Fack and Kosaki \cite{FK1986}*{Lemma 4.1} (cf. \cite{F1982}*{Lemma 3.3}). We note that the second statement of the lemma also follows from Remark \ref{rem:determined-by-spectral-distribution}.

\begin{lem} \label{lem:singular-numbers-3}
Let $\N$ be a diffuse finite von Neumann algebra with faithful normal trace $\tau$. For $T \in \N$ and $t \in [0,1]$,
\[
\int_0^t \mu_s(T)\, ds = \sup \{ \tau(|T|P) \mid P \in \Proj(\N),\ \tau(P) \leq t \}.
\]
In particular,
\[
\int^1_0 \mu_s(T) \, dt = \tau(|T|).
\]
\end{lem}

Let $\M$ be a \II1 factor, let $\A$ be a MASA in $\M$, and let $E_{\A} : \M \to \A$ denote the normal conditional expectation onto $\A$.  Hiai \cite{H1987}*{Theorem 4.5} (see also \cite{AK2006}*{Theorem 7.2}) showed that if $T \in \M$ is self-adjoint, then $E_{\A}(T) \maj T$. We require an extension of this result to the case when $T$ is not necessarily self-adjoint.

\begin{thm} \label{thm:expectation-submajorizes}
Let $\M$ be a \II1 factor, let $\N$ be a diffuse von Neumann subalgebra of $\M$, and let $E_{\N} : \M \to \N$ denote the normal conditional expectation onto $\N$.  For every $T \in \M$, $E_{\N}(T) \submaj T$.
\end{thm}

\begin{proof}
Let $\tau$ denote the faithful normal trace on $\M$. Observe that for $R,S \in \M$, expanding the inequality
\[
\tau((R^* - S)^*(R^* - S)) \geq 0
\]
implies
\begin{equation} \label{eq:expectation-submajorizes-ineq}
\re(\tau(RS)) \leq \frac{1}{2} \big(\tau(RR^*) + \tau(S^*S)\big).
\end{equation}

We will now prove that $|E_{\N}(T)| \submaj |T|$, which implies the desired result. Let $U \in \N$ be a unitary obtained from a polar decomposition of $E_{\N}(T)$, so that $|E_{\N}(T)| = U^*E_{\N}(T)$. Note for a projection $P \in \N$,
\begin{equation} \label{eq:expectation-submajorizes-1}
\tau(|E_{\N}(T)|P) = \tau(U^*E_{\N}(T)P) = \tau(E_{\N}(U^*TP)) = \tau(U^*TP).
\end{equation}

Let $W \in \M$ be a unitary obtained from a polar decomposition of $T$, so that $T = W|T|$. Then $V = U^*W$ is a unitary such that
\[
U^*T = U^*W|T| = V|T|.
\]
Given a projection $P \in \M$, applying (\ref{eq:expectation-submajorizes-ineq}) gives
\begin{align}
\tau(U^*TP) &= \tau(PV|T|P) \notag \\
&= \tau\big((PV|T|^{1/2})(|T|^{1/2}P)\big) \notag \\
&\leq \frac{1}{2}\big(\tau(PV|T|V^*P) + \tau(P|T|P)\big) \notag \\
&= \frac{1}{2}\big(\tau(|T|V^*PV) + \tau(|T|P)\big) \notag \\
&\leq \max \{ \tau(|T|Q),\ \tau(|T|P) \}, \label{eq:expectation-submajorizes-2}
\end{align}
where $Q = V^*PV$ is a projection with the same trace as $P$.

Given $t \in [0,1)$, applying (\ref{eq:expectation-submajorizes-1}), (\ref{eq:expectation-submajorizes-2}), and Lemma \ref{lem:singular-numbers-3} gives
\begin{align*}
\int_0^t \mu_s(|E_{\N}(T)|)\, ds &= \sup \{ \tau(|E_{\N}(T)|P) \mid P \in \Proj(\N), \ \tau(P) \leq t \} \\
&\leq \sup \{ \tau(|T|P) \mid P \in \Proj(\M),\, \tau(P) \leq t \} \\
&= \int_0^t \mu_s(|T|)\, ds
\end{align*}
and it follows that $|E_{\N}(T)| \submaj |T|$.
\end{proof}

\begin{rem}
Note that for $T$ as in the statement of Theorem \ref{thm:expectation-submajorizes}, $\tau(E_{\N}(T)) = \tau(T)$. Moreover, if $T$ is self-adjoint, then clearly $E_{\N}(T) \maj T$ if and only if $E_{\N}(T) + \alpha 1_{\M} \maj T + \alpha 1_{\M}$ for every $\alpha \geq 0$. Hence Remark \ref{rem:positives-majorization-submajorization-equivalent} and Theorem \ref{thm:expectation-submajorizes} implies \cite{AK2006}*{Theorem 7.2} of Arveson and Kadison.
\end{rem}

%%%%%%%%%%%%%%%%%%%%%%%%%%%%%%%%%%%%%%%%%%%
\subsection{Unitary orbits} \label{sec:orbits}
%%%%%%%%%%%%%%%%%%%%%%%%%%%%%%%%%%%%%%%%%%%
\begin{defn} \label{defn:unitary-orbits}
Let $\M$ be a von Neumann algebra and let $T \in \M$. 
\begin{enumerate}
\item The \emph{closed unitary orbit of $T$} is
\[
\O(T) := \overline{\{ UTU^* \mid U \in \U(\M) \}}.
\]
\item The \emph{closed two-sided unitary orbit of $T$} is
\[
\N(T) := \overline{\{ UTV \mid U,V \in \U(\M) \}}.
\]
\end{enumerate}
\end{defn}

Let $\M$ be a \II1 factor. It was shown by Kamei \cite{K1983}*{Theorem 4} (see also \cite{AK2006}*{Theorem 5.4}) that if $T \in \M$ is self-adjoint, then a self-adjoint element $S \in \M$ belongs to the closed unitary orbit $\O(T)$ if and only if $S$ and $T$ are equi-distributed. Hence, by Remark \ref{rem:determined-by-spectral-distribution}, $\O(T)$ is completely determined by the eigenvalue function of $T$.

The next theorem shows that for arbitrary $T \in \M$, the closed two-sided unitary orbit $\N(T)$ is completely determined by the singular value function of $T$.

\begin{thm} \label{thm:two-sided-unitary-orbit}
Let $\M$ be a \II1 factor and let $T \in \M$. Then
\[
\N(T) = \{ S \in \M \mid \mu_s(S) = \mu_s(T) \ \forall s \in [0,1) \}.
\]
\end{thm}
\begin{proof}
It follows from \cite{FK1986}*{Lemma 2.5} that the map taking $S \in \M$ to $\mu_s(S)$ is continuous for every $s \in [0,1)$. This implies the inclusion
\[
\N(T) \subset \{ S \in \M \mid \mu_s(S) = \mu_s(T) \ \forall s \in [0,1) \}.
\]

For the other inclusion, note that if $S \in \M$ satisfies $\mu_s(S) = \mu_s(T)$ for every $s \in [0,1)$, then $\mu_s(|S|) = \mu_s(|T|)$ for every $s \in [0,1)$. Hence, by Remark \ref{rem:equidistributed-equality-of-singular-values} and \cite{AK2006}*{Theorem 5.4}, $|S| \in \O(|T|)$ so there is a sequence of unitaries $U_n \in \M$ such that $\lim_n U_n|T|U_n ^*= |S|$.

Let $V$ and $W$ be unitaries obtained from polar decompositions of $S$ and $T$ respectively, so that $S = V|S|$ and $T = W|T|$. Then $W^*T = |T|$ which gives 
\[
\lim_n VU_nW^*TU_n = \lim_n VU_n|T|U_n^* = V|S| = S.
\]
Hence $S \in \N(T)$.
\end{proof}

\begin{rem}
We  note that for $T \in \M$, it follows directly from a polar decomposition of $T$ that $\N(T) = \N(|T|)$.
\end{rem}

%%%%%%%%%%%%%%%%%%%%%%%%%%%%%%%%%%%%%%%%%%%
\subsection{Non-increasing rearrangements} \label{sec:rearrangements}
%%%%%%%%%%%%%%%%%%%%%%%%%%%%%%%%%%%%%%%%%%%
For a real-valued function $f \in L^\infty([0,1),m)$, where $m$ denotes Lebesgue measure on $[0,1)$, the notion of the non-increasing rearrangement of $f$ will be essential in what follows (see e.g. \cite{HL1934}*{Section 10.12}, \cite{C1974} or \cite{HN1987}). It is convenient to work with the interval $[0,1)$ because this is the domain of the singular value function in a \II1 factor.

\begin{defn} \label{defn:rearrangement}
For a real-valued function $f \in L^\infty([0,1),m)$, the \emph{non-increasing rearrangement} of $f$ is the function
\[
f^*(s) = \inf \{ x \mid m(\{ t \mid f(t) \geq x \}) \leq s\}, \quad s \in [0,1).
\]
\end{defn}

Intuitively, the non-increasing rearrangement $f^*$ is obtained by ``rearranging'' the values of $f$. The function $f^*$ is non-increasing and right-continuous on $[0,1)$. Moreover, it is the unique function with these properties that is equi-distributed with $f$, in the sense of Section \ref{sec:eigenvalue-singular-value-functions}.

\begin{rem} \label{rem:iso-to-singular-value-function}
Let $\M$ be a \II1 factor with faithful normal trace $\tau$ and let $\A$ be a diffuse, abelian, countably generated von Neumann subalgebra of $\M$.  There is an isomorphism $\alpha : \A \to L^\infty([0,1),m)$ such that $\tau = \int_0^1dx \circ \alpha$. Let $A \in \A$ be self-adjoint, let $f = \alpha(A)$, and let $f^*$ denote the non-increasing rearrangement of $f$.  It follows from Definition \ref{defn:eigenvalue-singular-value-functions} and Definition \ref{defn:rearrangement} that $\lambda_s(A) = f^*(s)$ for every $s \in [0,1)$.  In addition, if $\A$ is generated by $e_A$, Remark \ref{rem:determined-by-spectral-distribution} implies one may select $\alpha$ such that $\alpha(A) = f^*$.
\end{rem}

The following three technical lemmas will be needed when we consider the relationship between the singular value function of an element in a \II1 factor and the singular value function of its compression to an invariant subspace.

\begin{lem} \label{lem:rearrangement-restriction}
Let $f \in L^\infty([0,1),m)$ be a real-valued function and let $f^*$ denote the non-increasing rearrangement of $f$. For every Borel subset $\X \subset [0,1)$, there is a Borel subset $\Y \subset [0,1)$ such that the restrictions $f|_{\Y}$ and $f^*|_{\X}$ are equi-distributed under an isomorphism of $L^\infty(\Y, m|_\Y)$ and $L^\infty(\X, m|_\X)$. The set $\Y$ is determined by $e_f(\X) = \mathbbm{1}_{\Y}$, where $e_f$ is the spectral measure from Proposition \ref{prop:special-spectral-measure}.
\end{lem}

\begin{proof}
From the above remarks, $\lambda_s(f) = f^*(s)$ for every $s \in [0,1)$. Hence, by Proposition \ref{prop:special-spectral-measure}, there is a spectral measure $e_f$ on $[0,1)$ such that $\tau(e_f([0,t))) = t$ for every $t \in [0,1)$ and
\[
f^n = \int_0^1 (f^*(t))^n\, de_f(t), \quad \forall n \geq 0.
\]

For a Borel set $\X \subset [0,1)$, there is a Borel set $\Y \subset [0,1)$ such that $e_f(\X) = \mathbbm{1}_{\Y}$, where $\mathbbm{1}_{\Y}$ denotes the indicator function corresponding to the set $\Y$. This implies that
\[
f^n \mathbbm{1}_{\Y} = \int_{\X} (f^*(t))^n\, de_f(t), \quad \forall n \geq 0.
\]
Thus
\[
\tau(f^n \mathbbm{1}_{\Y}) =  \int_{\X} (f^*(t))^n\, dt = \tau((f^*)^n \mathbbm{1}_{\X}), \quad \forall n \geq 0.
\]
It follows that $f|_{\Y}$ and $f^*|_{\X}$ are equi-distributed.
\end{proof}

\begin{lem} \label{lem:rearrangement-compression}
Let $f \in L^\infty([0,1),m)$ be a real-valued function and let $f^*$ denote the non-increasing rearrangement of $f$. 
Let $\Y$ and $\Z$ be Borel subsets determined as in Lemma \ref{lem:rearrangement-restriction} such that $f|_{\Y}$ and $f|_{\Z}$ are equi-distributed with $f^*|_{[0,1/2)}$ and $f^*|_{[1/2,1)}$ respectively. Then $(f|_{\Y})^*(s) = f^*(s/2)$  and $( f|_{\Z})^*(s) = f^*((s+1)/2)$ for all $s \in [0,1)$.
\end{lem}

\begin{proof}
We will only prove the result for the interval $[0,1/2)$, since the argument for the interval $[1/2,1)$ is similar. For brevity let $\X = [0,1/2)$. The functions $f|_{\Y}$ and $f^*|_\X$ are equi-distributed by assumption. Hence $(f|_{\Y})^* = (f^*|_\X)^*$, and it suffices to show that $(f^*|_\X)^*(s) = f^*(s/2)$ for all $s \in [0,1)$.

Let $\alpha : L^\infty(\X,m|_\X) \to L^\infty([0,1),m)$ denote the isomorphism defined for $g \in L^\infty(\X,m|_\X)$ by
\[
\alpha(g)(s) = g(s/2), \quad s \in [0,1).
\]
Consider the function $\alpha(f^*|_\X)$. This function is non-increasing and right continuous since $f^*$ is. Moreover, since $\alpha$ is an isomorphism, $\alpha(f^*|_\X)$ is equi-distributed with $f^*|_\X$. Hence by the uniqueness of the non-increasing rearrangement, $(f^*|_\X)^* = \alpha(f^*|_\X)$, which implies that for all $s \in [0,1)$,
\[
(f^*|_\X)^*(s) = \alpha(f^*|_\X)(s) = f^*|_\X(s/2) = f^*(s/2),
\]
as required.
\end{proof}

\begin{lem} \label{lem:rearrangement-range}
Let $f,g \in L^\infty([0,1),m)$ be real-valued functions. Let $X \subset [0,1)$ be a Borel set and let $\Y,\Z \subset [0,1]$ be Borel sets determined as in Lemma \ref{lem:rearrangement-restriction} such that $f|_{\Y}$ and $g|_{\Z}$ are equi-distributed with $f^*|_\X$ and $g^*|_\X$ respectively. Suppose that $f^*|_X \leq g^*|_\X$. Then $(f|_Y)^* \leq (g|_\Z)^*$.
\end{lem}

\begin{proof}
Since $f|_{\Y}$ and $g|_{\Z}$ are equi-distributed with $f^*|_\X$ and $g^*|_\X$ respectively, $(f|_{\Y})^* = (f^*|_\X)^*$ and $(g|_{\Z})^* = (g^*|_\X)^*$. Also, since $f^* \leq g^*$, it is immediately clear that $(f^*|_\X)^* \leq (g^*|_\X)^*$. It follows that $(f|_Y)^* \leq (g|_\Z)^*$.
\end{proof}

%%%%%%%%%%%%%%%%%%%%%%%%%%%%%%%%%%%%%%%%%%%
\section{Main Result} \label{sec:main}
%%%%%%%%%%%%%%%%%%%%%%%%%%%%%%%%%%%%%%%%%%%
\subsection{Outline of the proof}
%%%%%%%%%%%%%%%%%%%%%%%%%%%%%%%%%%%%%%%%%%%
In this section, we will prove our analogue of Thompson's theorem for \II1 factors. Before proceeding, we briefly outline our approach to the proof.

Let $\M$ be a \II1 factor, let $\A$ be a MASA in $\M$, and let $E_{\A} : \M \to \A$ denote the normal conditional expectation onto $\A$. Let $A \in \A$ and $T \in \M$ be elements satisfying $A \submaj T$. Our goal is to prove that we can find $S \in \N(T)$ such that $E_{\A}(S) = A$.

We will proceed in stages, beginning with very strong assumptions on $A$ and $T$, and then weakening these assumptions at each subsequent stage. Specifically, we will consider the following stages, listed in order of increasing generality:

\begin{enumerate}
\item The case of complete dominance, i.e.
\[
\sup \{ \mu_s(A) \mid s \in [0,1) \} \leq \inf \{ \mu_s(T) \mid s \in [0,1) \}.
\]
\item The case of strict dominance, i.e. for some constant $\delta > 0$, 
\[\mu_s(A) + \delta \leq \mu_s(T), \quad \forall s \in [0,1).\]
\item The case of dominance, i.e.
\[\mu_s(A) \leq \mu_s(T), \quad \forall s \in [0,1).\]
\item The general case, i.e. $A \submaj T$.
\end{enumerate}

Throughout the proof we will assume that $A$ and $T$ are positive. To see that there is no loss of generality in making this assumption, first note that $|A| \submaj |T|$ and $\N(T) = \N(|T|)$. Let $U \in \U(\A)$ be a unitary obtained from a polar decomposition of $A$, so that $A = U|A|$. If there is $S \in \N(T)$ such that $E_{\A}(S) = |A|$, then $U^*S \in \N(T)$ and $E_{\A}(U^*S) = U^*|A| = A$. In other words, proving the result for $|A|$ and $|T|$ implies the result for $A$ and $T$. 

%%%%%%%%%%%%%%%%%%%%%%%%%%%%%%%%%%%%%%%%%%%
\subsection{The case of complete dominance} \label{sec:complete-dominance}
%%%%%%%%%%%%%%%%%%%%%%%%%%%%%%%%%%%%%%%%%%%
\begin{lem} \label{lem:complete-dominance-1}
Let $\M$ be a \II1 factor with trace $\tau$ and let $\A$ be a MASA in $\M$. Let $A \in \A$ and $T \in \M$ be positive elements such that $A \submaj T$, $T$ is invertible, and
\begin{equation} \label{eq:complete-dominance-1-1}
\sup \{ \mu_s(A) \mid s \in [0,1) \} \leq \inf \{ \mu_s(T) \mid s \in [0,1) \}.
\end{equation}
Then there is a projection $P \in \Proj(\A)$ with $\tau(P) = 1/2$ and unitaries $U,V \in \U(\M)$ such that $PUTVP = AP$ and
\[
\sup \{ \mu_s(A') \mid s \in [0,1) \} \leq \inf \{ \mu_s(T') \mid s \in [0,1) \},
\]
where $A' = A P^\perp \in \A P^\perp$ and $T' = P^\perp T P^\perp \in P^\perp \M P^\perp$.
\end{lem}

\begin{proof}
Note $\mu_s(T) \neq 0$ for all $s \in [0,1)$ as $T$ is invertible.

Let $P = e_A([0,1/2))$, where $e_A$ is the projection-valued measure from Proposition \ref{prop:special-spectral-measure}. Then letting $A_1 = AP \in P \M P$ and $A_2 = AP^\perp \in P^\perp \M P^\perp$, it follows from Remark \ref{rem:iso-to-singular-value-function} and Lemma \ref{lem:rearrangement-compression} that 
\[
\mu_s(A_1) = \mu_{s/2}(A), \quad \forall s \in [0,1),
\]
and
\[
\mu_s(A_2) = \mu_{(1+s)/2}(A), \quad \forall s \in [0,1).
\]
Note that the singular values of $A_1$ and $A_2$ are precisely the singular values of $A$ along the intervals $[0,1/2)$ and $[1/2,1)$ respectively.

Let $Q = e_T([1/2,1))$. Then letting $T_1 = TQ \in Q \M Q$ and $T_2 = TQ^\perp \in Q^\perp \M Q^\perp$, we similarly have
\[
\mu_s(T_1) = \mu_{(1 + s)/2}(T), \quad \forall s \in [0,1),
\]
and
\[
\mu_s(T_2) = \mu_{s/2}(T), \quad \forall s \in [0,1).
\]
Note similarly that the singular values of $T_1$ and $T_2$ are precisely the singular values of $T$ along the intervals $[1/2,1)$ and $[0,1/2)$ respectively.

Since $\tau(P) = \tau(Q)$, there is a unitary $W \in \U(\M)$ conjugating $Q$ onto $P$ and $Q^\perp$ onto $P^\perp$. Let $S = WTW^* \in \M$.  Since $\tau(P) = 1/2$, if $\N = P \M P$ then $\M$ and $\M_2(\N)$ are isomorphic in such a way that we can write
\[
P = \left( \begin{matrix}
1_\N & 0 \\
0 & 0
\end{matrix} \right),
\quad
A = \left( \begin{matrix}
A_1 & 0 \\
0 & A_2
\end{matrix} \right),
\quad \text{and}
\quad
S = \left( \begin{matrix}
S_1 & 0 \\
0 & S_2
\end{matrix} \right),
\]
where $S_1, S_2 \in \N$ are positive elements such that
\[
\quad \mu_s(S_1) = \mu_s(T_1)
\quad \text{and}
\quad \mu_s(S_2) = \mu_s(T_2),
\]
for all $s \in [0,1)$.

Since $T$ is invertible, $S_1$ is invertible. Let
\[
H = A_1(S_1)^{-1} \in \N.
\]
Then $H$ is a contraction since the assumption of complete dominance (\ref{eq:complete-dominance-1-1}) implies
\begin{align*}
\| H \| &\leq \| A_1 \| \|(S_1)^{-1} \| \\
&= \mu_0(A) \inf \{ \mu_s(S_1) \mid s \in [0,1) \}^{-1} \\
&\leq 1.
\end{align*}

Let
\[
V = 
\left( \begin{matrix}
H & \sqrt{1-HH^*} \\
\sqrt{1 - H^*H} & -H^*
\end{matrix} \right).
\]
Then $V$ is a unitary since $H$ is a contraction.

Notice
\[
VS =
\left( \begin{matrix}
HS_1 & \ast \\
\ast & -H^*S_2
\end{matrix} \right)
=
\left( \begin{matrix}
A_1 & \ast \\
\ast & -H^*S_2
\end{matrix} \right).
\]
Since $VWTW^* = VS$, $PVSP = AP$ and $\mu_s(-H^*S_2) = \mu_s(H^*S_2)$ for every $s \in [0,1)$, it remains to show that
\[
\sup \{ \mu_s(A_2) \mid s \in [0,1) \} \leq \inf \{ \mu_s(H^*S_2) \mid s \in [0,1) \}.
\]

By construction,
\begin{align*}
HH^* &= A_1(S_1)^{-2}A_1 \\
&\geq \mu_0(S_1)^{-2}\mu_1(A_1)^2 1_{\M} \\
&= \mu_{1/2}(S)^{-2} \mu_{1/2}(A)^2 1_{\M},
\end{align*}
where we have used the fact that
\[
\mu_0(S_1) = \mu_0(T_1) = \mu_{1/2}(T) = \mu_{1/2}(S).
\]
Hence
\[
S_2 HH^* S_2 \geq \mu_{1/2}(S)^{-2} \mu_{1/2}(A)^2 S_2^2,
\]
which implies
\[
|H^* S_2| \geq \mu_{1/2}(S)^{-1} \mu_{1/2}(A) S_2.
\]

Since $\mu_s(|H^*S_2|) = \mu_s(H^*S_2)$ for every $s \in [0,1)$, it follows that
\begin{align*}
\mu_s(H^* S_2) &\geq \mu_{1/2}(S)^{-1} \mu_{1/2}(A) \mu_s(S_2) \\
&= \mu_{1/2}(S)^{-1} \mu_{1/2}(A) \mu_{s/2}(S) \\
&\geq \mu_{1/2}(A).
\end{align*}
Therefore, for $s,t \in [0,1)$,
\[
\mu_s(H^*S_2) \geq \mu_{(1+t)/2}(A) = \mu_t(A_2). \qedhere
\]
\end{proof}

\begin{lem} \label{lem:complete-dominance-2}
Let $\M$ be a \II1 factor with trace $\tau$, let $\A$ be a MASA in $\M$, let $E_{\A} : \M \to \A$ denote the normal conditional expectation onto $\A$, and let $T \in \M$ be positive. Then there are unitaries $U,V \in \U(\M)$ such that $E_{\A}(UVTV^*) = 0$. 
\end{lem}

\begin{proof}
Let $Q = e_T([0,1/2))$ and let $P \in \Proj(\A)$ be any projection with $\tau(P) = 1/2$. Since $\tau(P) = \tau(Q)$, there is a unitary $V \in \U(\M)$ that conjugates $Q$ onto $P$. Similarly, since $\tau(P) = \tau(P^\perp)$, there is a unitary $U \in \U(\M)$ that conjugates $P$ onto $P^\perp$. Observe that $PUP = P^\perp U P^\perp = 0$.

Since $T$ commutes with $Q$, $VTV^*$ commutes with $P$. Thus,
\begin{align*}
PE_{\A}(UVTV^*) &= PE_{\A}(UVTV^*)P \\
&= E_{\A}(PUVTV^*P) \\
&= E_{\A}(PUPVTV^*) \\
&= 0.
\end{align*}
Similarly, $P^\perp E_{\A}(UVTV^*) = 0$. Hence $E_{\A}(UVTV^*) = 0$.

\end{proof}

\begin{prop} \label{prop:complete-dominance}
Let $\M$ be a \II1 factor with trace $\tau$, let $\A$ be a MASA in $\M$, and let $E_{\A} : \M \to \A$ denote the normal conditional expectation onto $\A$. Let $A \in \A$ and $T \in \M$ be positive elements such that $A \submaj T$. If
\[
\sup \{ \mu_s(A) \mid s \in [0,1) \} \leq \inf \{ \mu_s(T) \mid s \in [0,1) \},
\]
then there are unitaries $U,V \in \M$ such that $E_{\A}(UTV) = A$.
\end{prop}

\begin{proof}
First, if $T$ is not invertible, then 
\[
\inf \{ \mu_s(T) \mid s \in [0,1) \} = 0,
\]
which would imply that $A = 0$. In this case, the result would follow immediately from Lemma \ref{lem:complete-dominance-2}.

Assume that $T$ is invertible. By Lemma \ref{lem:complete-dominance-1} there is a projection $P_1 \in \Proj(\A)$ with $\tau(P_1) = 1/2$, and unitaries $U_1,V_1 \in \U(\M)$ such that, setting $T_1 = U_1TV_1$, 
\[
P_1T_1P_1 = AP_1
\]
and
\[
\sup \{ \mu_s(A P_1^\perp) \mid s \in [0,1) \} \leq \inf \{ \mu_s(P_1^\perp T_1 P_1^\perp) \mid s \in [0,1) \}.
\]

Now consider $A P_1^\perp$ and $P_1^\perp T_1 P_1^\perp$. If $P_1^\perp T_1 P_1^\perp$ is not invertible, then as above, $A P_1^\perp = 0$, and we could apply Lemma \ref{lem:complete-dominance-2} to complete the argument. 

Otherwise, if $P_1^\perp T_1 P_1^\perp$ is invertible, then by Lemma \ref{lem:complete-dominance-1} there is a projection $P_2 \in \Proj(\A)$ with $P_2 \geq P_1$ and $\tau(P_2) = 3/4$, and unitaries $U_2,V_2 \in \U(\M)$ with $P_1U_2 = P_1 = U_2P_1$ and $P_1V_2 = P_1 = V_2P_1$ such that, setting $T_2 = U_2T_1V_2$,
\[
P_kT_2P_k = AP_k, \quad k=1,2,
\]
and
\[
\sup \{ \mu_s(A P_2^\perp) \mid s \in [0,1) \} \leq \inf \{ \mu_s(P_2^\perp T_2 P_2^\perp) \mid s \in [0,1) \}.
\]

We can repeat this process to obtain a sequence of projections $P_n \in \Proj(\A)$ with $P_{n+1} \geq P_n$ and $\tau(P_n) = 1 - 1/2^n$, and sequences of unitaries $U_n,V_n \in \U(\M)$ with $P_nU_{n+1} = P_n = U_{n+1}P_n$ and $P_nV_{n+1} = P_n = V_{n+1}P_n$ such that, setting $T_{n+1} = U_{n+1}T_nV_{n+1}$,
\[
P_kT_nP_k = AP_k, \quad 1 \leq k \leq n,
\]
and
\[
\sup \{ \mu_s(A P_n^\perp) \mid s \in [0,1) \} \leq \inf \{ \mu_s(P_n^\perp T_n P_n^\perp) \mid s \in [0,1) \}.
\]
If these sequences are finite, then the process terminates after an application of Lemma \ref{lem:complete-dominance-2} and the argument is complete.

Suppose these sequences are infinite. Then since
\[
\|1 - U_{n+1}\|_2 = \|P_n^\perp - P_n^\perp U_{n+1} P_n^\perp\|_2 \leq 2\|P_n^\perp\|_2 \leq 1/2^{n-1},
\]
it is easy to verify that the sequences of unitaries $(U_n \cdots U_1)_{n=1}^\infty$ and $(V_1 \cdots V_n)_{n=1}^\infty$ are Cauchy in the $2$-norm on $\M$. Hence they converge strongly to unitaries $U,V \in \M$ respectively.

Now for every $k \geq 1$,
\begin{align*}
P_kE_{\A}(UTV) &= \sotlim_{n \to \infty} P_kE_{\A}(T_n)P_k \\
&= \sotlim_{n \to \infty} E_{\A}(P_kT_nP_k) \\
&= \sotlim_{n \to \infty}(AP_k) \\
&= AP_k. 
\end{align*}
%\[
%P_kE_{\A}(UTV) = \sotlim_{n \to \infty} P_kE_{\A}(T_n)P_k = \sotlim_{n \to \infty}(AP_k) = AP_k. 
%\]
It follows that $E_{\A}(USV) = A$.
\end{proof}

%%%%%%%%%%%%%%%%%%%%%%%%%%%%%%%%%%%%%%%%%%%
\subsection{The case of strict dominance} \label{sec:strict-dominance}
%%%%%%%%%%%%%%%%%%%%%%%%%%%%%%%%%%%%%%%%%%%

\begin{lem} \label{lem:strict-dominance}
Let $\M$ be a \II1 factor and let $\A$ be a MASA in $\M$. Let $A \in \A$ and $T \in \M$ be positive elements and suppose there is a constant $\delta > 0$ such that $\mu_s(A) + \delta \leq \mu_s(T)$ for all $s \in [0,1)$. Then there are countably many disjoint intervals $(I_n)_{n=1}^\infty$ with $I_n = [a_n, b_n) \subset [0,1)$ and $a_n < b_n$ such that $\cup_{n \geq 1} I_n = [0,1)$ and 
\[
\sup \{ \mu_s(A) \mid s \in I_n \} \leq \inf \{ \mu_s(T) \mid s \in I_n \}, \qquad n\geq 1.
\]
\end{lem}

\begin{proof}
For $a,b \in [0,1)$ with $a < b$, we will temporarily say that the interval $[a,b)$ is {\em good} if
\[
\sup \{ \mu_s(A) \mid s \in [a,b) \} \leq \inf \{ \mu_s(T) \mid s\in [a,b) \}.
\]
Let $\X$ denote the collection of all families $\F$ consisting of disjoint good subintervals of $[0,1)$ such that $\cup_{I \in \F} I = [0,c)$ for some $c \in (0,1)$.

By the right-continuity of $\mu_t(A)$ and $\mu_t(T)$, combined with the fact that $\mu_t(A) + \delta \leq \mu_t(S)$ for all $t \in [0,1)$, it follows easily that there is $d \in (0,1)$ such that the interval $[0,d)$ is good. Hence $\X$ is non-empty.

Order $\X$ by inclusion and let $(\F_\lambda)_{\lambda \in \Lambda}$ be an increasing chain of families in $\X$. Taking $\F_0 = \cup_{\lambda \in \Lambda} \F_\lambda$, it follows immediately that $\F_0$ is an upper bound of $(\F_\lambda)_{\lambda \in \Lambda}$ in $\X$. Hence we can apply Zorn's Lemma to obtain a maximal family $\F_m \in \X$.

Write $\cup_{I \in \F_m} I = [0,c)$.  If $c < 1$, then we can argue as before to find a good interval $[c, d) \subset [0,1)$. Since the family $\F_m \cup \{ [c,d) \}$ would belong to $\X$, this would contradict the maximality of $\F_m$. Hence $c = 1$.

We conclude by observing that the family $\F_m$ is countable since the length of each interval in $\F_m$ is non-zero by construction.
\end{proof}

\begin{prop} \label{prop:strict-dominance}
Let $\M$ be a \II1 factor with trace $\tau$, let $\A$ be a MASA in $\M$, and let $E_{\A} : \M \to \A$ denote the normal conditional expectation onto $\A$. Let $A \in \A$ and $T \in \M$ be positive elements and suppose there is a constant $\delta > 0$ such that $\mu_s(A) + \delta \leq \mu_s(T)$ for all $s \in [0,1)$. Then there are unitaries $U,V \in \M$ such that $E_{\A}(UTV) = A$.
\end{prop}

\begin{proof}
Let $(I_n)_{n=1}^\infty$ be a countable family of disjoint intervals constructed as in Lemma \ref{lem:strict-dominance}. For each $n \geq 1$, let $P_n = e_A(I_n)$ and $Q_n = e_T(I_n)$. Since $\tau(P_n) = \tau(Q_n)$, there is a unitary $W \in \U(\M)$ that conjugates $Q_n$ onto $P_n$ for each $n \geq 1$. Let $S = WTW^*$ and note that since $T$ commutes with each $Q_n$, $S$ commutes with each $P_n$. 

For each $n \geq 1$, let $\M_n = P_n \M P_n$, $\A_n = \A P_n$, $A_n = AP_n \in \A_n$, and $S_n = SP_n \in \M_n$. Then for each $n \geq 1$, the construction of $A_n$ and $S_n$ combined with Remark \ref{rem:iso-to-singular-value-function} and Lemma  \ref{lem:rearrangement-range} implies that
\[
\sup \{ \mu_s(A_n) \mid s \in [0,1) \} \leq \inf \{ \mu_s(S_n) \mid s \in [0,1) \}.
\]
Hence by Proposition \ref{prop:complete-dominance}, for each $n \geq 1$ there are unitaries $U_n,V_n \in \U(\M_n)$ such that $E_{\A_n}(U_nT_nV_n) = A_n$. Letting $U = \oplus_{n = 1}^\infty U_n$ and $V = \oplus_{n = 1}^\infty V_n$, it follows from above that $E_{\A}(UWTW^*V) = A$.
\end{proof}

%%%%%%%%%%%%%%%%%%%%%%%%%%%%%%%%%%%%%%%%%%%
\subsection{The case of dominance}
%%%%%%%%%%%%%%%%%%%%%%%%%%%%%%%%%%%%%%%%%%%

The following lemma is an immediate consequence of Remark \ref{rem:iso-to-singular-value-function} and Lemma \ref{lem:rearrangement-range}.

\begin{lem} \label{lem:dominance}
Let $\M$ be a \II1 factor with trace $\tau$ and let $A,T \in \M$ be positive elements. Let $\X \subset [0,1)$ be a Borel set with the property that there is $\delta \geq 0$ such that
\[
\mu_s(A) + \delta \leq \mu_s(T), \quad \forall s \in \X.
\]
Let $P = e_A(\X)$ and $Q = e_T(\X)$. Let $A' = AP$ and $T' = TQ$. Then
\[
\mu_s(A') + \delta \leq \mu_s(T'), \quad \forall s \in [0,1).
\]
\end{lem}

%\begin{proof}
%Let $\A$ and $\B$ be diffuse, abelian, countably generated von Neumann subalgebras of $\M$ containing $A$ and $T$ respectively (and hence the spectral projections of $A$ and $T$ respectively). Let $\alpha : \A \to L^\infty([0,1],m)$ and $\beta : \B \to L^\infty([0,1],m)$ be isomorphisms such that $\tau|_\A = \int_0^1dm \circ \alpha$ and $\tau|_\B = \int_0^1dm \circ \beta$. Let $f = \alpha(A)$ and $g = \beta(T)$.
%
%Let $f^*$ and $g^*$ denote the non-increasing rearrangements of $f$ and $g$ as in Definition \ref{defn:rearrangement}. As in the proof of Lemma \ref{lem:rearrangement-restriction}, there are Borel subsets $\Y_A,\Y_T \subset [0,1]$ such that $\alpha(A') = f \mathbbm{1}_{\Y_A}$ and $\alpha(T') = g \mathbbm{1}_{\Y_T}$, and such that $f \mathbbm{1}_{\Y_A}$ and $g \mathbbm{1}_{\Y_T}$ are equi-distributed with $f^* \mathbbm{1}_{\X}$ and $g^* \mathbbm{1}_{\X}$ respectively.
%
%By equi-distributivity, $\mu_s(A') = \mu_s(f \mathbbm{1}_{\Y_A}) = \mu_s(f^* \mathbbm{1}_{\X})$ and $\mu_s(T') = \mu_s(g \mathbbm{1}_{\Y_T}) = \mu_s(g^* \mathbbm{1}_{\X})$ for every $s \in [0,1)$. Since $\mu_s(A) = f^*(s)$ and $\mu_s(T)= g^*(s)$ for every $s \in [0,1)$, it follows that $f^* \mathbbm{1}_{\X} + \delta \leq g^* \mathbbm{1}_{\X}$ on $\X$. Hence
%\[
%\mu_s(A') + \delta = \mu_s(f^* \mathbbm{1}_{\X}) + \delta \leq \mu_s(g^* \mathbbm{1}_{\X}) = \mu_s(T'), \quad \forall s \in [0,1).  \qedhere
%\]
%\end{proof}

\begin{prop} \label{prop:dominance}
Let $\M$ be a \II1 factor, let $\A$ be a MASA in $\M$, and let $E_{\A} : \M \to \A$ denote the normal conditional expectation onto $\A$. Let $A \in \A$ and $T \in \M$ be positive elements with the property that $\mu_s(A) \leq \mu_s(T)$ for every $s \in [0,1)$.  Then there exists an $S \in \N(T)$ such that $E_\A(S) = A$.
\end{prop}

\begin{proof}
Let
\[
\X_0 = \{s \in [0,1) \mid \mu_s(T) = \mu_s(A)\}
\] 
and, for each $n \geq 1$, let
\[
\X_n = \{ s \in [0,1) \mid \left\|T\right\|/(n+1) < \mu_s(T) - \mu_s(A) \leq \left\|T\right\|/n \}.
\]
Note that $\bigcup_{n\geq 0} \X_n = [0,1)$ as $\mu_s(A) \leq \mu_s(T) \leq \left\|T\right\|$ for all $x \in [0,1)$.  Furthermore, each $\X_n$ is Borel since the functions $s \to \mu_s(A)$ and $s \to \mu_s(T)$ are right-continuous on $[0,1)$.

For each $n \geq 0$ let $P_n = e_A(\X_n)$ and $Q_n = e_T(\X_n)$. Since $\tau(P_n) = \tau(Q_n)$, and since the sets $\{P_n \mid n \geq 0\}$ and $\{Q_n \mid n \geq 0\}$ are each pairwise orthogonal, there is a unitary $W \in \U(\M)$ that conjugates each $Q_n$ onto $P_n$.

Let $T' = WTW^*$ and, for each $n \geq 0$, let $\M_n = P_n\M P_n$, $\A_n = A_nP_n$, $A_n = AP_n \in \A_n$, and $S_n = T'P_n \in \M_n$. Note by Lemma \ref{lem:dominance} that $\mu_s(A_0) = \mu_s(S_0)$ for all $s \in [0,1)$ and, for each $n\geq 1$
\[
\mu_{s}(A_n) + \|T\|/(n+1) \leq \mu_{s}(S_n), \quad \forall s \in [0,1).
\]
Hence $A_0$ and $S_0$ are equi-distributed, and thus $A_0 \in \O(S_0)$ in $\M_0$ by \cite{AK2006}*{Theorem 5.4}.  Furthermore, by Proposition \ref{prop:strict-dominance}, there are unitaries $U_n,V_n \in \M_n$ such that $E_{\A_n}(U_nS_nV_n) = A_n$.

Letting $U = \oplus_{n = 1}^\infty U_n \in P_0^\perp \M P_0^\perp$ and $V = \oplus_{n = 1}^\infty V_n \in P_0^\perp \M P_0^\perp$, which are unitaries, it follows from above that $E_{\A}(UP_0^\perp WTW^* P_0^\perp V) = AP_0^\perp$. Taking $S = UP_0^\perp WTW^*P_0^\perp V \oplus A_0 \in \M$, it follows that $S \in \N(T)$ and $E_{\A}(S) = A$.
\end{proof}

\begin{rem} \label{rem:precise-unitaries-in-dominance}
In the proof of Proposition \ref{prop:dominance}, if we make the additional assumption that the set
\[
\{ \mu_s(T) \mid s \in [0,1),\ \mu_s(T) = \mu_s(A) \}
\]
is finite, then $S_0$ and $A_0$ take on only a finite number of singular values, and it is elementary to show that there exists a unitary $U_0 \in \M$ such that $U_0S_0U_0^* = A_0$. Letting $U' = \oplus_{n = 0}^\infty U_n$ and $V' = U_0^* \oplus_{n = 1}^\infty V_n$, which are unitaries in $\M$, it follows that $E_{\A}(U'WTW^*V') = A$. Thus it is possible to construct unitaries $U, V \in \M$ such that $E_{\A}(UTV) = A$.
\end{rem}

\begin{rem}
If one is willing to forego the possibility of constructing explicit unitaries $U,V \in \M$ such that $E_{\A}(UTV) = A$ as in Remark \ref{rem:precise-unitaries-in-dominance}, then there is a slightly quicker proof of Proposition \ref{prop:dominance}, which we now sketch.

Let $\A_0$ denote the diffuse abelian subalgebra of $\A$ generated by $e_A$ as defined in Proposition \ref{prop:special-spectral-measure}. By Remark \ref{rem:iso-to-singular-value-function}, there is an isomorphism $\alpha : \A_0 \to L^\infty([0,1),m)$ such that $\tau = \int_0^1 dx \circ \alpha$ and $\alpha(A)(s) = \mu_s(A)$. Let $f,g,h \in L^\infty([0,1),m)$ denote the functions defined for $s \in [0,1)$ by $f(s) = \mu_s(A)$, $g(s) = \mu_s(T)$, and
\[
h(s) := 
\begin{cases}
\mu_s(A) / \mu_s(T) & \mu_s(T) \ne 0, \\
0 & \text{otherwise}.
\end{cases}
\]
Observe that the assumptions of Proposition \ref{prop:dominance} imply that $0 \leq f \leq g$. Hence $f = gh$.

Let $T_0 = \alpha^{-1}(g) \in \A_0$. Then $\mu_s(T_0) = \mu_s(T)$ for every $ s \in [0,1)$, so $T_0 \in \O(T)$ by Remark \ref{rem:equidistributed-equality-of-singular-values} and \cite{AK2006}*{Theorem 5.4}. Similarly, let $B = \alpha^{-1}(h) \in \A_0$. It is clear that $B$ is a positive contraction such that
\[
T_0B = \alpha^{-1}(gh) = \alpha^{-1}(f) = A.
\]

Choose $\beta \in [0,1]$ such that $\tau(B) = 2\beta - 1$, and let
\[
U_0 = \alpha^{-1}(\mathbbm{1}_{[0,\beta)} - \mathbbm{1}_{[\beta,1)}) \in \M.
\]
It is not difficult to check that $U_0$ is a self-adjoint unitary satisfying $B \prec U_0$. By Remark \ref{rem:equidistributed-equality-of-singular-values} and Theorem \ref{thm:ravichandran}, there is a self-adjoint unitary $U \in \O(U_0)$ such that $E_{\A}(U) = B$. Letting $S = UT_0 \in \N(T)$ gives
\[
E_{\A}(S) = E_{\A}(UT_0) = E_{\A}(U)T_0 = BT_0 = A.
\]
\end{rem}

%%%%%%%%%%%%%%%%%%%%%%%%%%%%%%%%%%%%%%%%%%%
\subsection{The general case}
%%%%%%%%%%%%%%%%%%%%%%%%%%%%%%%%%%%%%%%%%%%
We are now ready to prove our analogue of Thompson's theorem for \II1 factors. 

\begin{theorem*}
Let $\M$ be a \II1 factor, let $\A$ be a MASA in $\M$, let $E_{\A} : \M \to \A$ denote the normal conditional expectation onto $\A$, let $T \in \M$, and let $A \in \A$.  Then $A \submaj T$ if and only if there exists $S \in \N(T)$ such that $E_{\A}(S) = A$. 
\end{theorem*}

\begin{proof}
If there exists an $S \in \N(T)$ such that $E_\A(S) = A$, then $A \submaj S$ by Theorem \ref{thm:expectation-submajorizes}.  Since, by Theorem \ref{thm:two-sided-unitary-orbit}, $\mu_s(S) = \mu_s(T)$ for all $s \in [0,1)$, we obtain that $A \submaj T$ so one direction is complete.

For the other direction, note we may assume that $A$ and $T$ are positive by the discussion at the end of Section \ref{sec:main}. Let 
\[
\X = \{ s \in [0,1) \mid \mu_s(T) - \mu_s(A) \leq 0 \},
\]
and let $\Y = [0,1) \setminus \X$. Define the function
\[
f(t) := \int_{\X} (\mu_s(T) - \mu_s(A))\, ds + \int_{\Y\cap[0,t)} (\mu_s(T) - \mu_s(A))\, ds, \quad t \in [0,1].
\]
Notice that $f$ is continuous on $[0,1]$, $f(0) \leq 0$ by construction, and $f(1) \geq 0$ since $A \submaj T$. Hence there is $t_0 \in [0,1]$ such that $f(t_0) = 0$. Let $\Z = \X \cup (\Y \cap [0,t_0))$, and note that $\Z = [0,t_0) \cup (\X \setminus [0,t_0))$.

Note if $m(\Z) = 0$, then $m(\X) = 0$, and thus the result is complete by Proposition \ref{prop:dominance}.  Thus we may assume that $m(\Z) \neq 0$.

Let $P = e_A(\Z)$, $Q = e_T(\Z)$, $A_1 = AP$, $A_2 = AP^\perp$, $T_1 = TQ$, and $T_2 = TQ^\perp$. Note that, by construction and Lemma \ref{lem:dominance}, we have $\mu_s(A_2) \leq \mu_s(T_2)$ for every $s \in [0,1)$.  

Since $\tau(P) = \tau(Q)$, there is a unitary $U \in \M$ that conjugates $Q$ onto $P$. Let $T' = UTU^*$, let $T'_1 = T'P$, and let $T'_2 = T'P^\perp$. Thus $\mu_s(T'_1) = \mu_s(T_1)$ and $\mu_s(A_2) \leq \mu_s(T_2) = \mu_s(T'_2)$ for every $s \in [0,1)$.

We claim $A_1 \maj T'_1$. To see this, we consider two possibilities for $t \in [0,1]$. If $t \in [0, t_0/m(\Z))$, then applying Lemma \ref{lem:singular-numbers-3} combined with the fact that  $A \submaj T$ implies
\[
\int_0^t (\mu_s(T'_1) - \mu_s(A_1))\, ds = \frac{1}{m(\Z)}\int_0^{m(\Z)t} (\mu_s(T) - \mu_s(A))\, ds \geq 0.
\]
Otherwise, if $t \in [t_0/m(\Z), 1]$ then applying Lemma \ref{lem:singular-numbers-3} again combined with the fact that $\mu_s(T) - \mu_s(A) \leq 0$ on $\X$ implies
\begin{align*}
\int_0^t  (\mu_s(T'_1) - \mu_s(A_1))\, ds 
&\geq \frac{1}{m(\Z)}\int_0^{t_0} (\mu_s(T) - \mu_s(A))\, ds \\
&\quad + \frac{1}{m(\Z)}\int_{\X \setminus [0,t_0)} (\mu_s(T) - \mu_s(A))\, ds \\
&= \frac{1}{m(\Z)}f(t_0) \\
&=0,
\end{align*}
with equality when $t = 1$.  Hence $A_1 \maj T'_1$.

We can now apply Theorem \ref{thm:ravichandran} to $A_1$ and $T'_1$, considered as elements of $P \M P$, to obtain $S_1 \in \O(T'_1)$ such that $E_{\A P}(S_1) = A_1$. We can also apply Proposition \ref{prop:dominance} to $A_2$ and $T_2$, considered as elements of $P^\perp \M P^\perp$ to obtain $S_2 \in \N(T'_2)$ such that $E_{\A P^\perp}(S_2) = A_2$. Taking $S = S_1 \oplus S_2 \in \M$, it follows that $S \in \N(T)$ and $E_{\A}(S) = A$.
\end{proof}

\begin{rem}
Let $A$ and $T$ be as in the statement of Theorem \ref{thm:thompson-intro}. If the singular value functions of $A$ and $T$ each take only a finite number of values, then in the proof of Theorem \ref{thm:thompson-intro} we can apply Remark \ref{rem:precise-unitaries-in-dominance} in addition to Proposition \ref{prop:dominance}, and apply Bhat and Ravichandran's Schur-Horn theorem for self-adjoint elements with finite spectrum \cite{BR2011}*{Theorem 3.1} instead of Ravichandran's Schur-Horn theorem, to obtain unitaries $U,V \in \U(\M)$ such that $E_{\A}(UTV) = A$.  Alternatively, it is elementary to verify that $\{UTV \mid U, V \in \U(\M)\}$ is closed when the singular value function of $T$ takes only a finite number of values, and thus Theorem \ref{thm:thompson-intro} implies there exists unitaries $U,V \in \U(\M)$ such that $E_{\A}(UTV) = A$. 
\end{rem}
%%%%%%%%%%%%%%%%%%%%%%%%%%%%%%%%%%%%%%%%%%%
\section{Thompson's theorem and the Schur-Horn theorem} \label{sec:relationship}
%%%%%%%%%%%%%%%%%%%%%%%%%%%%%%%%%%%%%%%%%%%
The proof of Thompson's theorem for \II1 factors given in the previous section depends on Ravichandran's Schur-Horn theorem for \II1 factors. However, in this section we will prove that, logically, Thompson's theorem for \II1 factors implies Ravichandran's Schur-Horn theorem.

We are grateful to David Sherman for showing us a simple proof of the following lemma.

\begin{lem} \label{lem:sherman}
Let $\M$ be a \II1 factor with trace $\tau$ and let $S,T \in \M$. If $T$ is positive, $\mu_s(S) = \mu_s(T)$ for every $s \in [0,1)$, and $\tau(S) = \tau(T)$, then $S$ is positive.
\end{lem}

\begin{proof}
Let $U$ be a unitary obtained from a polar decomposition of $S$, so that $S = U|S|$. Then
\[
\tau(|S|) = \int_0^1 \mu_s(|S|)\, ds = \int_0^1 \mu_s(S)\, ds = \int_0^1 \mu_s(T) \, ds = \tau(T) = \tau(S).
\]
Thus, by the Cauchy-Schwarz inequality,
\[
\tau(S) = \tau(|S|) = \tau(U|S|^{1/2} |S|^{1/2}) \leq \tau(|S|)^{1/2} \tau(|S|)^{1/2} = \tau(|S|).
\]
Since this inequality is actually an equality, it follows that $U|S|^{1/2}$ is a complex scalar multiple of $|S|^{1/2}$, and hence that $S$ is a complex scalar multiple of $|S|$. Since $\tau(S) = \tau(|S|)$, this implies that $S = |S|$.
\end{proof}

\begin{thm}
Thompson's theorem for \II1 factors implies Ravichandran's Schur-Horn theorem for \II1 factors.
\end{thm}

\begin{proof}
Let $\M$ be a \II1 factor with trace $\tau$, let $\A$ be a MASA in $\M$, and let $E_{\A} : \M \to \A$ denote the normal conditional expectation onto $\A$. Let $A \in \A$ and $T \in \M$ be self-adjoint elements satisfying $A \maj T$.  As the conclusions of the Schur-Horn theorem are invariant under translations by real-valued scalars, we may assume that $A$ and $T$ are positive.  By Thompson's theorem for \II1 factors (Theorem \ref{thm:thompson-intro}), there is $S \in \N(T)$ such that $E_{\A}(S) = A$. The result will follow if we can show that $S \in \O(T)$.

Note that $\mu_s(S) = \mu_s(T)$ for every $s \in [0,1)$. Also, $\tau(T) = \tau(A) = \tau(S)$, since $E_{\A}$ preserves the trace. Hence by Lemma \ref{lem:sherman}, $S$ is positive and it follows from Remark \ref{rem:equidistributed-equality-of-singular-values} and \cite{AK2006}*{Theorem 5.4} that $S \in \O(T)$.
\end{proof}
%%%%%%%%%%%%%%%%%%%%%%%%%%%%%%%%%%%%%%%%%%%

\end{document}